\documentclass[12pt]{article}
\usepackage{amsmath, amsthm}\usepackage{enumerate}
\usepackage{amssymb}\usepackage{bbold}
\usepackage{color}
\newtheorem{theorem}{Theorem}
\newtheorem{example}{Example}[section]

\newtheorem{proposition}{Proposition}[section]
\newtheorem{lemma}{Lemma}[section]
\newtheorem{corollary}{Corollary}[section]

\newtheorem{definition}{Definition}[section]

\DeclareMathOperator{\convo}{\xrightarrow[]{\mathbb{o}}}
\DeclareMathOperator{\convuo}{\xrightarrow[]{\mathbb{uo}}}
\DeclareMathOperator{\convbbuo}{\xrightarrow[]{\mathbb{bbuo}}}
\DeclareMathOperator{\convw}{\xrightarrow[]{\mathbb{w}}}

\begin{document}
\title{{\bf Bibounded $uo$-convergence and $b$-property in vector lattices}}\maketitle\author{\centering{{Safak Alpay$^{1}$, Eduard Emelyanov$^{1}$, Svetlana Gorokhova $^{2}$\\ \small $1$ 
Middle East Technical University, Ankara, Turkey\\ \small $2$ Southern Mathematical Institute of the Russian Academy of Sciences, Vladikavkaz, Russia\abstract{We define bidual bounded $uo$-convergence in vector lattices and investigate relations between this convergence and $b$-property. We prove that for a regular Riesz dual system $\langle X,X^{\sim}\rangle$, $X$ has $b$-property if and only if the order convergence in $X$ agrees with  the order convergence in $X^{\sim\sim}$.}

{\bf{keywords:}} {\rm vector lattice, order dual, regular Riesz dual system, $b$-property, unbounded order convergence, Banach lattice}
\vspace{2mm}


\section{Introduction and preliminaries}

In the present paper, all vector lattices are supposed to be real and Archime\-dean. 
By $X^{\sim}$ we denote the order dual of a vector lattice $X$ and by $X^{\sim}_n$ 
its order continuous dual. A pair $\langle X,Y\rangle$ is called a {\em Riesz dual system} 
if $Y$ is an order ideal of $X^{\sim}$ separating points of $X$ \cite[Def.3.51]{AB2}. 
The natural duality in $\langle X,Y\rangle$ is $\langle x,y\rangle:=y(x)$.
For any Riesz dual system $\langle X,Y\rangle$, $X$ will be identified with its image
$\hat{X}\subseteq X^{\sim\sim}$ under the canonical embedding $x\to i(x)=\hat{x}$,
where  $\hat{x}(y):=y(x)$ for $y\in X^{\sim}$. For a Riesz dual system $\langle X,Y\rangle$, 
it is well known that $\hat{X}$ is a vector sublattice of $Y^{\sim}_n$ and hence of $Y^{\sim}$ 
(cf. \cite[p.173]{AB2}). 

\begin{definition}\label{regular Riesz dual system}
A  Riesz dual system $\langle X,Y\rangle$ is called regular if $\hat{X}$ is a regular sublattice of $Y^{\sim}$.
\end{definition}

The following proposition can be considered as a supplement to Theorem 3.54 of \cite{AB2}. 

\begin{proposition}\label{regular embedding}
For a Riesz dual system $\langle X,Y\rangle$, the following statements are equivalent.
\begin{enumerate}
\item[$i)$] \  $\langle X,Y\rangle$ is a regular Riesz dual system.
\item[$ii)$] \ $\hat{X}$ is a regular sublattice of $Y^{\sim}_n$.
\item[$iii)$] \  $Y\subseteq X^{\sim}_n$.
\item[$iv)$] \ $\hat{X}$ is an order dense sublattice of $Y^{\sim}_n$.
\end{enumerate}
\end{proposition}

\begin{proof}\
$i)\Longrightarrow ii)$: \ It follows from $\hat{X}\subseteq Y^{\sim}_n$ because $\hat{X}$ 
is a regular sublattice of $Y^{\sim}$ in view of $i)$.

$ii)\Longrightarrow iii)$: \ Let $x_\alpha\downarrow 0$ in $X$. Then $\hat{x}_\alpha\downarrow 0$ in $\hat{X}$,  and 
since $\hat{X}$ is a regular sublattice of $Y^{\sim}_n$ then also $\hat{x}_\alpha\downarrow 0$ in $Y^{\sim}_n$.
Hence $y(x_\alpha)=\hat{x}_\alpha(y)\to 0$ for all $y\in Y$ e.g. by \cite[Thm.1.67]{AB1}. 
It follows that each $y\in Y$ is order continuous, as desired. 

$iii)\Longleftrightarrow iv)$ \ is contained in Theorem 3.54 of \cite{AB2}.

$iv)\Longrightarrow ii)$ \ is Theorem 1.23 of \cite{AB1}.

$ii)\Longrightarrow i)$: \ Since $\hat{X}$ is a regular sublattice of $Y^{\sim}_n$ and $Y^{\sim}_n$, being a band in $Y^{\sim}$, is a regular sublattice of $Y^{\sim}$
then $\hat{X}$ is a regular sublattice of $Y^{\sim}$. 
\end{proof}

\noindent
It is worth to mention that Lemma 3.2 of \cite{GLX} follows directly from the equivalence $iii)\Longleftrightarrow vi)$ of Proposition \ref{regular embedding}.

Let $X^{\sim}$ separate points of $X$. If $X^{\sim}_n=X^{\sim}$ then $\langle X,X^{\sim}\rangle$ is a {\em regular Riesz dual system}
and hence $\hat{X}$ is a regular sublattice of $X^{\sim\sim}$. The next fact follows now from Proposition \ref{regular embedding}

\begin{corollary}\label{regular embedding 2}
For a Riesz dual system $\langle X,X^{\sim}\rangle$ the following conditions are equivalent$:$
\begin{enumerate}
\item[$i)$] \  $X^{\sim}_n=X^{\sim}$$.$
\item[$ii)$] \ $\hat{X}$ is a regular sublattice of $X^{\sim\sim}$.
\end{enumerate}
\end{corollary}

A net $x_\alpha$ in a vector lattice $X$ is unbounded order convergent (briefly, $uo$-convergent) to $x\in X$, 
whenever $|x_\alpha-x|\wedge u\convo 0$ for every $u\in X_+$. For any net $x_\alpha$ in $X$, we have:
\begin{enumerate}
\item[$(a)$] \ $x_\alpha\convuo 0$ iff $|x_\alpha|\convuo 0$;  
\item[$(b)$] \ $x_\alpha\convo 0$ iff $x_\alpha\convuo 0$ and $x_\alpha$ is eventually order bounded. 
\end{enumerate}
In particular, a functional $y\in X^{\sim}$ belongs to $X^{\sim}_n$ iff $y(x_\alpha)\to 0$ for any order bounded net $x_\alpha$ such that $x_\alpha\convuo 0$.

An important case of regular Riesz dual systems was introduced and investigated recently in \cite{GLX};
namely $\langle X,X^{\sim}_{uo}\rangle$, with $X^{\sim}_{uo}$ separating points of a normed lattice $X$. 
Any functional $y\in X^{\sim}$ which takes $uo$-null nets to null nets is a linear combination of the 
coordinate functionals of finitely many atoms of $X$; see, e.g. \cite[Prop.2.2]{GLX}. Therefore the usual 
way of defining the $uo$-dual of $X$ fails to be interesting. In order to make the definition meaningful,
for the case when $X$ is a normed lattice, Gao, Leung, and Xanthos set an additional condition on $uo$-null nets. 
Namely, they define $X^{\sim}_{uo}$ as the collection of all functionals from $X^{\sim}$ taking norm bounded $uo$-null 
nets to null nets \cite[Def.2.1]{GLX}. 

In the case of an arbitrary vector lattice, the first candidate for such an additional condition, the {\em eventually order 
boundedness} of $uo$-convergent nets fails again because it turns $uo$--convergent nets to just $o$--convergent nets. 

Assuming $X^{\sim}$ separates the points of $X$, we investigate another additional condition, namely the 
{\em eventually order boundedness of $uo$-null nets in $X^{\sim\sim}$}. Recall that a subset $A$ of $X$ is called {\em $b$-order bounded} 
whenever $\hat{A}$ is order bounded in $X^{\sim\sim}$ \cite[Def.1.1]{AAT1}; $X$ has {\em $b$-property}, whenever every $b$-order bounded 
subset of $X$ is order bounded \cite[Def.1.1]{AAT1}.

\begin{definition}\label{bbuo-convergence}
Let $X^{\sim}$ separate  points of $X$.  A net $x_\alpha$ in $X$ is called $bbuo$-convergent to $x\in X$ if 
$x_\alpha\convuo x$ and the net $\hat{x}_\alpha$ is eventually order bounded in $X^{\sim\sim}$.
\end{definition}

Note that, in $C[0,1]$, the norm bounded $uo$-convergence, $bbuo$-convergence, and $o$-convergence agree.
In particular, the $bbuo$-convergence is not topological \cite[Thm.2]{G} (see also \cite[Thm.2.2]{DEM}). 

Clearly, any $o$-convergent net is $bbuo$-convergent and, by Lemma \ref{b-property 1}, in the case when $X$ has $b$-property, 
$bbuo$-convergence agrees with $o$-conver\-gence. Since every order dual vector lattice $X=Y^{\sim}$ has $b$-property,
replacement of eventually $b$-order boundedness by eventually order boundedness in $2n$-th order dual of $X$ for some $n\in\mathbb{N}$ 
leads to the same notion as the eventually $b$-order boundedness (the case $n=1$).

In the present paper we investigate relations between $bbuo$-convergence and $b$-property in vector and Banach lattices.
For further unexplained terminology and notations we refer to \cite{AB1,AB2,AAT2,GTX}.

\section{$bbuo$-Convergence in vector lattices}

In this section, we assume that $X^{\sim}$ separates points of the vector lattice $X$, so that $\langle X,X^{\sim}\rangle$ is a Riesz dual system.
We begin with the following two lemmas.

\begin{lemma}\label{any positive is o-cont}
For a Riesz dual system $\langle X,X^{\sim}\rangle$ the following conditions are equivalent$:$
\begin{enumerate}
\item[$i)$] \  $\hat{X}$ is a regular sublattice of $X^{\sim\sim}$.
\item[$ii)$] \  $x_\alpha\convbbuo x$ implies $\hat{x}_\alpha\convo \hat{x}$ in $X^{\sim\sim}$ for every net $x_\alpha$ in $X$ and $x\in X$.
\item[$iii)$] \  $X^{\sim}_n=X^{\sim}$.
\end{enumerate}
\end{lemma}

\begin{proof}\
$i)\Longrightarrow ii)$: \  Let $x_\alpha\convbbuo x$. Then $x_\alpha\convuo x$ in $X$ and hence $\hat{x}_\alpha\convuo \hat{x}$ in $X^{\sim\sim}$ \cite[Thm.3.2]{GTX}. 
Since the net $\hat{x}_\alpha$ is eventually order bounded in $X^{\sim\sim}$, $\hat{x}_\alpha\convo \hat{x}$ in $X^{\sim\sim}$

$ii)\Longrightarrow i)$: \ Let $x_\alpha\downarrow 0$ in $X$. Then $x_\alpha\convbbuo 0$ and, by the assumption, $\hat{x}_\alpha\convo 0$ in $X^{\sim\sim}$.
Hence $\hat{x}_\alpha\downarrow 0$ in $X^{\sim\sim}$. By Lemma 2.5 of \cite{GTX}, $\hat{X}$ is a regular sublattice of $X^{\sim\sim}$.

$i)\Longleftrightarrow iii)$ is Corollary \ref{regular embedding 2}
\end{proof}

\begin{lemma}\label{b-property 1}
Let $x_\alpha$ be a net in a vector lattice $X$ possessing $b$-property, $x\in X$. Then $x_\alpha\convbbuo x$ iff $x_\alpha\convo x$.
\end{lemma}

\begin{proof}\
It suffices to show that $x_\alpha\convbbuo 0$ implies $x_\alpha\convo 0$. Let $x_\alpha\convbbuo 0$ in $X$. 
Hence $|\hat{x}_\alpha|\le u\in X^{\sim\sim}$ for all $\alpha\ge\alpha_0$. Since $X$ has $b$-property, we may assume  $u=\hat{w}\in \hat{X}$. 
So, $x_\alpha\convuo 0$ in $X$ and $|x_\alpha|\le w\in X$ for all $\alpha\ge\alpha_0$. Thus $x_\alpha\convo 0$ in $X$.
\end{proof}

We define the $bbuo$-dual by $X^{\sim}_{bbuo}:=\{y\in X^{\sim}|x_\alpha\convbbuo 0\Rightarrow y(x_\alpha)\to 0\}$.
Since $x_\alpha\convo 0\Rightarrow x_\alpha\convbbuo 0$ then $X^{\sim}_{bbuo}\subseteq X^{\sim}_n$. Clearly, $X^{\sim}_{bbuo}$ is an order ideal in $X^{\sim}_n$
and hence in $X^{\sim}$. Furthermore, in the case of a normed lattice $X$, both $X^{\sim}_{uo}$ and $X^{\sim}_{bbuo}$ are clearly norm closed ideals in $X^{\sim\sim}$, 
and $X^{\sim}_{uo}\subseteq X^{\sim}_{bbuo}$. We include several simple examples. 

\begin{example}\label{duals}
\begin{enumerate}
\item[$(a)$] \ $(c_0)^{\sim}_{uo}=\ell^1$, $(\ell^1)^{\sim}_{uo}=c_0$, and $(\ell^\infty)^{\sim}_{uo}=\ell^1$ $\cite[Ex.2.4]{GLX}$$;$ 
therefore $((c_0)^{\sim}_{uo})^{\sim}_{uo}=c_0$, $((\ell^\infty)^{\sim}_{uo})^{\sim}_{uo}=c_0$, and $((\ell^1)^{\sim}_{uo})^{\sim}_{uo}=\ell^1$. 
\item[$(b)$] \ $(c_0)^{\sim}_{bbuo}=\ell^1$, $(\ell^1)^{\sim}_{bbuo}=\ell^\infty$, and $(\ell^\infty)^{\sim}_{bbuo}=\ell^1$$;$ therefore
$((\ell^\infty)^{\sim}_{bbuo})^{\sim}_{bbuo}=\ell^\infty$, $((c_0)^{\sim}_{bbuo})^{\sim}_{bbuo}=\ell^\infty$, and  $((\ell^1)^{\sim}_{bbuo})^{\sim}_{bbuo}=\ell^1$.
\item[$(c)$] \ Let $X$ be an atomic universally complete vector lattice, without lost of generality $X=s(\Omega)$ the space of real-valued functions on a set $\Omega$. 
Then $X^{\sim}_{bbuo}=X^{\sim}_n=X^{\sim}=c_{00}(\Omega)$ the space of all real-valued functions on $\Omega$ with finite support, and   
$c_{00}(\Omega)^{\sim}_{uo}=c_{00}(\Omega)^{\sim}_{bbuo}=c_{00}(\Omega)^{\sim}_n=s(\Omega)$. Therefore $(X^{\sim}_{bbuo})^{\sim}_{bbuo}=X$ and 
$((c_{00}(\Omega)^{\sim}_{bbuo})^{\sim}_{bbuo}=c_{00}(\Omega)$.
\item[$(d)$] \ Let $(\Omega,\Sigma,P)$ be a non-atomic probability space, $1\le p\le \infty$, and $L^p:= L^p(\Omega,\Sigma,P)$.
Then $(L^p)^{\sim}_{uo}=L^q$ for $1<p\le\infty$, $q^{-1}+p^{-1}=1$$;$ and $(L^1)^{\sim}_{uo}=\{0\}$ $\cite{GLX}$. On the other hand, 
$(L^p)^{\sim}_{bbuo}=L^q$ for all $1\le p\le\infty$.
\end{enumerate}
\end{example}

The following result states that $X^{\sim}_{bbuo}$ indeed coincides with $X^{\sim}_n$. In particular, 
the duality theory for $bbuo$-convergence is already well presented in the literature. 

\begin{theorem}\label{b-property 3}
Let $X^{\sim}_n$ separate  points of $X$. Then $X^{\sim}_{bbuo}=X^{\sim}_{n}$.
\end{theorem}

\begin{proof}\
It is enough to prove that $X^{\sim}_{n}\subseteq X^{\sim}_{bbuo}$.
Let $X\ni x_\alpha\convbbuo 0$ and $y\in X^{\sim}_n$. We have to show $y(x_\alpha)\to 0$.
Without lost of generality, we assume $x_\alpha\ge 0$ for all $\alpha$. By Proposition \ref{regular embedding}, 
$\hat{X}$ is a regular sublattice of $(X^{\sim}_n)^{\sim}_n$ and hence of $(X^{\sim}_n)^{\sim}$.
Since $\hat{x}_\alpha\convuo 0$ in $\hat{X}$ then $\hat{x}_\alpha\convuo 0$ in $(X^{\sim}_n)^{\sim}$ 
by Theorem 3.2 of \cite{GTX}. Take $z\in X^{\sim\sim}$ with $0\le\hat{x}_\alpha\le z$ for all $\alpha$. 
Denoting by the same letters $\hat{x}_\alpha$ and $z$ their restrictions to $X^{\sim}_n$, 
gives $0\le\hat{x}_\alpha\le z$ in $(X^{\sim}_n)^{\sim}$ for all $\alpha$. 
So, the net $\hat{x}_\alpha$ is order bounded in $(X^{\sim}_n)^{\sim}$ and since $\hat{x}_\alpha\convuo 0$ 
in $(X^{\sim}_n)^{\sim}$ then $\hat{x}_\alpha\convo 0$ in $(X^{\sim}_n)^{\sim}$.
Since $\hat{y}\in ((X^{\sim}_n)^{\sim})^{\sim}_{n}$, then $\hat{y}(\hat{x}_\alpha)\to 0$ and hence
$y({x}_\alpha)=\hat{x}_\alpha(y)=\hat{y}(\hat{x}_\alpha)\to 0$ as desired.
\end{proof}

\noindent
We do not know whether the statement of Theorem \ref{b-property 3} still holds true under the weaker condition for $X^{\sim}$ to separate points of $X$. 

Recall that a vector lattice $X$ is said to be {\em perfect} if $\langle X,X^{\sim}_n\rangle$ is a regular Riesz dual system and 
$\hat{X}=(X^{\sim}_n)^{\sim}$. By the Nakano theorem \cite[Thm.3.18]{AB1}, the order dual $X^{\sim}$ of any vector lattice is perfect. 
By Theorem \ref{b-property 3}, $X^{\sim}_{bbuo}=X^{\sim}_{n}$ and hence $X^{\sim}_{bbuo}$ is also perfect as a projection band in $X^{\sim}$. 

\begin{lemma}\label{b-Cauchy}
Let $\langle X,X^{\sim}\rangle$ be a Riesz dual system such that for every net $x_\alpha$ in $X$$:$ $\hat{x}_\alpha\convo 0$ in $X^{\sim\sim}$ implies $x_\alpha\convo 0$ in $X$.
If $y_\alpha$ is a net in $X$ such that $\hat{y}_\alpha\convo z$ in $X^{\sim\sim}$ then $y_\alpha$ is Cauchy in $X$.
\end{lemma}

\begin{proof}\
Let $y_\alpha$ be a net in $X$ satisfying $\hat{y}_\alpha\convo z$ in $X^{\sim\sim}$. Then $\hat{y}_\alpha$ is Cauchy in $X^{\sim\sim}$. Therefore 
the double net $\hat{y}_{\alpha'}-\hat{y}_{\alpha''}$ $o$-converges to 0 in $X^{\sim\sim}$. By the conditions of the lemma, 
$y_{\alpha'}-y_{\alpha''}\convo 0$ in $X$ as desired.
\end{proof}

The following result characterizes $b$-property in terms of $bbuo$-convergence.

\begin{theorem}\label{b-property}
For a regular Riesz dual system $\langle X,X^{\sim}\rangle$ the following conditions are equivalent$:$
\begin{enumerate}
\item[$i)$] \  $X$ has $b$-property.
\item[$ii)$] \  $x_\alpha\convbbuo 0$ implies $x_\alpha\convo 0$ for every net $x_\alpha$ in $X$$.$
\item[$iii)$] \  $\hat{x}_\alpha\convo 0$ in $X^{\sim\sim}$ implies $x_\alpha\convo 0$ in $X$ for every net $x_\alpha$ in $X$$.$
\item[$iv)$] \  $\hat{x}_\alpha\convo 0$ in $X^{\sim\sim}$ iff $x_\alpha\convo 0$ in $X$ for every net $x_\alpha$ in $X$$.$
\end{enumerate}
\end{theorem}

\begin{proof}\ 
$i)\Longrightarrow ii)$ \ follows from Lemma \ref{b-property 1}.

$ii)\Longrightarrow iii)$: \ If $\hat{x}_\alpha\convo 0$ in $X^{\sim\sim}$ then $\hat{x}_\alpha\convuo 0$ in $X^{\sim\sim}$ and, 
by regularity of $\hat{X}$ in $X^{\sim\sim}$, $x_\alpha\convuo 0$ in $X$. By the assumption, $x_\alpha$ is eventually order bounded 
in $X^{\sim\sim}$ and hence $ii)$ implies $x_\alpha\convo 0$ in $X$, as desired. 

$iii)\Longrightarrow i)$: \ Let $X_+\ni x_\alpha\uparrow$ and $\hat{x}_\alpha\le u\in X^{\sim\sim}$ for all $\alpha$. 
We need to show that, for some $x\in X$, there holds  $x_\alpha\le x$  for all $\alpha$.
Since  $X^{\sim\sim}$ is Dedekind complete, $\hat{X}_+\ni\hat{x}_\alpha\uparrow\le u\in X^{\sim\sim}$ implies 
$\hat{x}_\alpha\convo z$ in $X^{\sim\sim}$ for some $z\in X^{\sim\sim}$. By Lemma \ref{b-Cauchy}, $x_\alpha$ is a Cauchy net in $X$.
Then there exists a net $y_\gamma$ in $X$ with $y_\gamma\downarrow 0$ in $X$ such that for every $\gamma$ there exists  
$\alpha_{\gamma}$ satisfying $|x_{\alpha'}-x_{\alpha''}|\le y_\gamma$ whenever $\alpha',\alpha''\ge \alpha_{\gamma}$. 
Fix any $\gamma_0$ and take $\alpha_{\gamma_0}$ such that $|x_{\alpha'}-x_{\alpha''}|\le y_{\gamma_0}$ for all 
$\alpha',\alpha''\ge \alpha_{\gamma_0}$. In particular, $x_{\alpha}-x_{\alpha_{\gamma}}\le y_{\gamma_0}$ for all $\alpha\ge \alpha_{\gamma_0}$ and hence 
$x_{\alpha}\le x:=x_{\alpha_{\gamma_0}}+y_{\gamma_0}$ for all $\alpha$ as desired.

$iii)\Longrightarrow iv)$: \ In view of $iii)\Longleftrightarrow ii)$, we need to prove that $x_\alpha\convo 0$ in $X$ 
implies $\hat{x}_\alpha\convo 0$ in $X^{\sim\sim}$. This follows from  regularity of $\hat{X}$ in $X^{\sim\sim}$.

$iv)\Longrightarrow iii)$ \ is trivial.
\end{proof}

The condition that every disjoint sequence $x_n$ in $X$ which is order bounded in $X^{\sim\sim}$ is also order bounded in $X$
does not imply the $b$-property. To see this, consider the first example at page 2 in \cite{AAT2}, the Banach lattice $X=\ell^\infty_\omega(T)$ 
consisting of all countably supported real functions on an uncountable set $T$. Clearly $X$ failed to have $b$-property. However 
$X$ has the countable $b$-property in the sense of \cite[p.2]{AAT2}. In particular, every sequence $x_n$ in $X$ which is order bounded 
in $X^{\sim\sim}$ is also order bounded in $X$.

\section{$bbuo$-Convergence in Banach lattices}

In this section, we consider the Banach lattice case. We begin with the following characterization of $KB$-spaces, which
extends Proposition 2.1 of \cite{AAT1}, where the equivalence $2)\Longleftrightarrow 3)$ was proved. 

\begin{theorem}\label{KB-spaces}
Let $X$ be a Banach lattice with order continuous norm. The following conditions are equivalent.
\begin{enumerate}
\item[$1)$] \ $X$ is perfect.
\item[$2)$] \ $X$ is a $KB$-space. 
\item[$3)$] \ $X$ has $b$-property.
\item[$4)$] \ $|x_n|\ \underset{w}\convuo \ 0$ implies $\|x_n\|\to 0$ for every sequence $x_n$ in $X$.
\end{enumerate}
If $X^{\sim}_{uo}$ separates points of $X$, then the above conditions are also equivalent to the following$:$
\begin{enumerate}
\item[$5)$] \ $X=Y^*$ for some Banach lattice $Y$.
\end{enumerate}
\end{theorem}

\begin{proof}\
The implication $1)\Longrightarrow 2)$ follows from the Nakano characterization of 
perfect vector lattices \cite[Thm.1.71]{AB2} utilizing the order continuity of the norm in $X$.

$2)\Longrightarrow 1)$: \  We apply Theorem 1.71 of \cite{AB2} once more, taking in account 
that $X^{\sim}_n=X^*$ separates the points of $X$ due to order continuity of the norm in $X$.
So,  let $X\ni x_\alpha\uparrow$ and $\sup_\alpha f(x_\alpha)<\infty$ for each $f\in (X^{\sim}_n)_+$. 
Then $\sup_\alpha f(x_\alpha)<\infty$ for each $f\in X^{\sim}_n=X^*$. 
The uniform boundedness principle ensures that the set $\{x_\alpha\}_\alpha$ is norm bounded.
Since $X$ is a $KB$-space, we derive that $\|x_\alpha-x\|\to 0$ for some $x\in X$. Since $x_\alpha\uparrow$ 
and $\|x_\alpha-x\|\to 0$ then $x_\alpha\uparrow x$ and hence $X$ is perfect.

$2)\Longleftrightarrow 3)$: \ This is Proposition 2.1 of \cite{AAT1}.

$2)\Longrightarrow 4)$: \ 
Let $X$ be a $KB$-space and $x_n$ be a sequence with $|x_n|\ \underset{w}\convuo\ 0$.
Since $X$ has order continuous norm, for each $\varepsilon>0$ and $x\in X_+$, there exists $y'\in X^*_+$ 
with $(|x'|-y')^+ (x)<\varepsilon $ for all $x'\in X^*$ with $\|x'\|\leq 1$ (see, e.g. \cite[Thm.4.18]{AB2}). 
So, for each $\varepsilon>0$, $x_n$, and each $0\leq x'$, $\|x'\|\leq 1$, we have
$$
   x'(|x_n|)\leq [x'\wedge y'](|x_n|)+(|x'|-y')^+(|x_n|)\leq y'(|x_n|)+\varepsilon.
$$
Therefore
$$
   \|x_n\|=\||x_n|\|\leq\sup\{x'(|x_n|): 0\leq x', \ \|x'\|\leq 1\}\leq y'(|x_n|)+\varepsilon.
$$
As $|x_n|\convw 0$, $\lim\sup\|x_n\|\leq 2\varepsilon$ for each $\varepsilon>0$, and hence $\|x_n\|\to 0$.

$4)\Longrightarrow 2)$: \
Since the norm in $X$ is order continuous, for proving that $X$ is a $KB$-space, it is enough to show that
$\|x_n\|\to 0$ for each disjoint sequence $x_n$ satisfying $0\leq|x_n|\leq x''$ for some $x''\in X^{**}$.
Let a sequence $x_n$ in $X$ be disjoint and $0\leq|x_n|\leq x''\in X^{**}$ for all $n$. For each $x'\in X^*_+$, we have
$$
    \sum^{m}_{n=1}x'(|x_n|)=x'\left(\sum^{m}_{n=1}|x_n|\right)=x'\left(\bigvee^{m}_{n=1}|x_n|\right)\leq x'(x''),
$$
and hence $\sum\limits^{\infty}_{n=1}x'(|x_n|)<\infty$. Therefore $x_n\convw 0$. 
Since each disjoint sequence in $X$ is $uo$-null, it follows from $4)$ that $\|x_n\|\to 0$.

$5)\Longrightarrow 1)$: \  This holds since the order dual of every vector lattice is  perfect 
(cf. \cite[Ex.3,p.74]{AB2}) and since $X=Y^*=Y^{\sim}$. Note that proving  
this implication we did not use that $X^{\sim}_{uo}$ separates points of $X$.

$2)\Longrightarrow 5)$: \  
The proof of this implication is just a combination of several results of paper \cite{GLX}. 
By Theorem 2.3 of \cite{GLX}, $X^{\sim}_{uo}$ is the Banach lattice $(X^{\sim}_n)^a$ that is 
the order continuous part of $X^{\sim}_n$. Since $X$ is a $KB$-space, $X$ is monotonically complete. 
Applying Theorem 3.4 of \cite{GLX} gives that $X$ is lattice isomorphic to the dual space 
$(X^{\sim}_{uo})^*$ under $i(x)(y)=y(x)$ for $x\in X$, $y\in X^{\sim}_{uo}$. Since both $X$ and 
$(X^{\sim}_{uo})^*$ are Banach lattices, the bijection $i:X\to (X^{\sim}_{uo})^*$ is also a homeomorphism. 
As it was pointed out in \cite{GLX} after the proof of \cite[Thm.3.4]{GLX}, $i$ is an isometry 
iff the closed unit ball $B_X$ is order closed. The later is clearly true since $X$ is a $KB$-space. 
So, $X$ is lattice isometric to the dual space $(X^{\sim}_{uo})^*$.
\end{proof}

\noindent
Notice that the condition 4) of Theorem \ref{KB-spaces} cannot be replaced by$:$  
\begin{enumerate}
\item[$4')$] \ for every sequence $x_n$ in $X$, $x_n\ \underset{w}\convuo\ 0$ implies $\|x_n\|\to 0$,
\end{enumerate}
because, due to Theorem 3.11 of \cite{GX}, the condition $4')$ is equivalent to the positive Schur property, 
which is, in general, stronger than $KB$.

Clearly $X^{\sim}_{uo}$ separates points of $X$ if the Banach lattice $X$ is atomic. Another case when 
$X^{\sim}_{uo}$ separates points of $X$ is described in Lemma 2.2 of \cite{TL}. Taking these 
two cases together, we get immediately from Proposition \ref{KB-spaces} the following characterization. 

\begin{corollary}\label{KB-spaces for atomic BL and for ri-space}
Let $X$ be a Banach lattice with order continuous norm. If $X$ is either atomic or else a rearrangement 
invariant space on a non-atomic probability space such that $X$ is not an $AL$-space 
then the following conditions are equivalent.
\begin{enumerate}
\item[$1)$] \ $X$ is perfect.
\item[$2)$] \ $X$ is a $KB$-space. 
\item[$3)$] \ $X$ has $b$-property.
\item[$4)$] \ $X=Y^*$ for some Banach lattice $Y$.
\end{enumerate}
Furthermore, in this case, $X$ is lattice isometric to $(X^{\sim}_{uo})^*$. 
\end{corollary}

The following result is similar to Theorem 2.3 of \cite{GLX}, that characterizes the dual of $X^{\sim}_{uo}$. 
Unlike in Theorem \ref{b-property 3}, $X^{\sim}_n$ is not required to be separating points of $X$. 

\begin{theorem}\label{bbuo dual}
Let $y$ be an order continuous functional on a Banach lattice $X$. The following conditions are equivalent$:$
\begin{enumerate}
\item[$1)$] \ $y\in X^{\sim}_{bbuo}$. 
\item[$2)$] \ $y(x_n)\to 0$ for each $b$-bounded  $uo$-null sequence $x_n$ in $X$.
\item[$3)$] \ $y(x_n)\to 0$ for each $b$-bounded disjoint sequence in $X$.
\end{enumerate}
\end{theorem}

\begin{proof}
$1)\Longrightarrow 2) \Longrightarrow 3)$ \ are clear. 

$3)\Longrightarrow 1)$: \ Suppose $y\in X^{\sim}_n$ satisfies 
$y(x_n)\to 0$ for each eventually $b$-bounded disjoint sequence in $X$. 
Let $x_\alpha$ be an eventually $b$-bounded and $uo$-null net in $X$. 
We show $y(x_\alpha)\to 0$. Without lost of generality, we assume the net $x_\alpha$ 
to be $b$-bounded itself, say $-z\leq \hat{x}_\alpha\leq z\in X^{\sim\sim}$. 
Let $A$ be the solid hull of $[-z,z]\cap\hat{X}$ in $X^{\sim\sim}$. 
Clearly, $A\subseteq [-z,z]$. Each disjoint sequence in $A$ is a disjoint 
sequence in $[-z,z]$ and therefore weakly converges to zero. 
So we see that, for each disjoint sequence $x_n$ in $A$, $|y|(x_n)\to 0$. 
Now applying this observation to Theorem 4.36 of \cite{AB2} for the norm 
continuous seminorm $p(x)=|y|(|x|)$, we see that, for $\varepsilon>0$, 
there exists $u\in X_+$ such that
$$
  p((|x|-u)^+)<\varepsilon  \ \ \   \text{ for all} \ \ \  x\in A.
$$
Hence, 
$$
  \sup_{x\in A}|y|\left(|x|-|x|\wedge u\right)=\sup_{x\in A}|y|\left(|x|-u\right)^+\leq \varepsilon.
$$
Recalling the equality $|x|=|x|\wedge u +(|x|-u)^+$ and utilizing the fact $x_\alpha\convuo 0$, 
from $|x_\alpha|\wedge u\convo 0\Rightarrow |y|(|x_\alpha|\wedge u)\to 0$ we have
$$  
  \lim\sup_\alpha |y|(|x_\alpha|)\to 0.
$$
Since $\varepsilon$ is arbitrary, $|y(x_\alpha)|\leq|y|(|x_\alpha|)\to 0$ as desired.
\end{proof}

{\tiny 

}
\end{document}